\documentclass{amsart}

\usepackage[all]{xy}
\usepackage{amssymb}
\usepackage{amsthm}
\usepackage[headings]{fullpage}

\newcommand{\QQ}{\mathbb Q}
\newcommand{\ZZ}{\mathbb Z}

\newcommand{\BB}{\mathbb B}
\newcommand{\LL}{\mathcal{L}}
\newcommand{\CC}{\mathbb C}
\newcommand{\Qp}{\QQ_p}
\newcommand{\Zp}{\ZZ_p}

\newcommand{\calO}{\mathcal{O}}
\newcommand{\DD}{\mathbb D}
\newcommand{\vp}{\varphi}
\newcommand{\Dcris}{\DD_\mathrm{cris}}

\newcommand{\ff}{\mathfrak{f}}
\newcommand{\K}{\mathcal{K}}
\newcommand{\Hh}{\mathbb{H}}
\newcommand{\Qpn}{\mathbb{Q}_{p,n}}
\DeclareMathOperator{\Char}{Char}
\DeclareMathOperator{\Sel}{Sel}
\renewcommand{\a}{\mathfrak{a}}

\DeclareMathOperator{\Tw}{Tw}
\DeclareMathOperator{\per}{per}
\DeclareMathOperator{\Sym}{Sym}
\DeclareMathOperator{\Iw}{Iw}
\DeclareMathOperator{\Gal}{Gal}

\DeclareMathOperator{\cris}{cris}
\DeclareMathOperator{\Fil}{Fil}

\DeclareMathOperator{\loc}{loc}
\newcommand{\OO}{\mathcal{O}}

\newcommand{\symV}{\Sym^2(V_f)}
\newcommand{\HH}{\mathcal{H}}
\newcommand{\RR}{\mathbb R}
\DeclareMathOperator{\col}{Col}
\DeclareMathOperator{\Ind}{Ind}
\DeclareMathOperator{\cor}{cor}
\newcommand{\HIw}{\mathbb{H}^1_{\Iw}}
\newcommand{\IndQK}{\Ind_K^\QQ}
\newtheorem{theorem}{Theorem}[section]
\newtheorem{proposition}[theorem]{Proposition}
\newtheorem{lemma}[theorem]{Lemma}
\newtheorem{corollary}[theorem]{Corollary}
\newtheorem{remark}[theorem]{Remark}
\newtheorem{definition}[theorem]{Definition}
\newtheorem{hypothesis}[theorem]{Hypothesis}

\begin{document}

\title{Iwasawa Theory for the Symmetric Square of a CM Modular Form at Inert Primes}

\author{Antonio Lei}
\address{School of Mathematical Sciences\\
Monash University\\
Clayton, VIC 3800\\
Australia}
\email{antonio.lei@monash.edu}
\thanks{The author is supported by an ARC DP1092496 grant.}

\begin{abstract}
Let $f$ be a CM modular form and $p$ an odd prime which is inert in the CM field. We construct two $p$-adic $L$-functions for the symmetric square of $f$, one of which has the same interpolating properties as the one constructed by Delbourgo-Dabrowski, whereas the second one has a similar interpolating properties but corresponds to a different eigenvalue of the Frobenius. The symmetry between these two $p$-adic $L$-functions allows us to define the plus and minus $p$-adic $L$-functions \`{a} la Pollack. We also define the plus and minus $p$-Selmer groups analogous to Kobayashi's Selmer groups. We explain how to relate these two sets of objects via a main conjecture.
\end{abstract}


\maketitle

 
\section{Introduction}
Let $f$ be a normalised eigen-newform of weight $k$, level $N$ and character $\epsilon$. Fix a prime $p\ne2$ such that $p\nmid N$. In \cite{delbourgodabrowski} (also in \cite{coatesschmidt} under some additional conditions), even distributions on $\Zp^\times$ are constructed to interpolate the $L$-values of the symmetric square of $f$. More precisely, if the Euler factor of $L(E,s)$ at $p$ is given by $(1-\alpha_1(p)p^{-s})(1-\alpha_2(p)p^{-s})$, then there exists an admissible distribution $\mu_{\alpha_i(p)^2}$ for $i=1,2$ such that 
\begin{equation}\label{inter}
\int_{\Zp^\times}\theta d\mu_{\alpha_i(p)^2}=\frac{p^{3n(k-1)}}{\alpha_i(p)^{2n}\tau(\theta^{-1})}\times\frac{L(\Sym^2f,\theta^{-1},2k-2)}{(\text{period})}
\end{equation}
for any non-trivial even Dirichlet character $\theta$ of conductor $p^n$ where $\tau(\theta^{-1})$ denotes the Gauss sum of $\theta^{-1}$.

Since the Euler factor of $L(\Sym^2f,s)$ at $p$ is $(1-\alpha_1(p)^2p^{-s})(1-\alpha_2(p)^2p^{-s})(1-\epsilon(p)p^{k-1-s})$, we expect that there should be a distribution $\mu_{\epsilon(p)p^{k-1}}$ satisfying interpolating properties similar to \eqref{inter}, but with $\alpha_i(p)^2$ replaced by $\epsilon(p)p^{k-1}$. In this paper, we construct such a distribution for the case when $f$ is a CM modular form that is non-ordinary at $p$. In other words, when the $L$-function of $f$ coincides with that of a Grossencharacter $\phi$ defined over $K$ and $p$ inerts in $K$. More precisely, we prove the following theorem in \S\ref{padicl} (Theorem~\ref{prelog}).

\begin{theorem}\label{exist}
If $f$ is as above, then there exist even admissible distributions $\mu_{\pm\epsilon(p)p^{k-1}}$ such that
$$
\int_{\Zp^\times}\theta d\mu_{\pm\epsilon(p)p^{k-1}}=\frac{p^{3n(k-1)}}{\left(\pm\epsilon(p)p^{k-1}\right)^{n}\tau(\theta^{-1})}\times\frac{L(\Sym^2f,\theta^{-1},2k-2)}{\rm(period)}.
$$
\end{theorem}

Note that we have $\alpha_1(p)^2=\alpha_2(p)^2=-\epsilon(p)p^{k-1}$ in this case,  methods in \cite{delbourgodabrowski} only produce one distribution, which agrees with $\mu_{-\epsilon(p)p^{k-1}}$ as given by Theorem~\ref{exist}.

The idea of the construction is rather simple. Let $V_f$ be the $p$-adic representation of $G_\QQ$ associated to $f$ as constructed by Deligne in \cite{deligne69}. In order to prove Theorem~\ref{exist}, we make use of the following observation. As $G_\QQ$-representations, we have
$$
\symV\cong V_1\oplus V_2
$$
where $V_1$ is an one-dimensional representation associated to some Dirichlet character $\eta$ twisted by a power of the cyclotomic character and $V_2$ is a two-dimensional representation associated to the Grossencharacter $\phi^2$. This implies that the $L$-function of $f$ factorises into
$$
L(\Sym^2f,s)=L(\phi^2,s)L(\eta,s-k+1).
$$
We can therefore make use of an Euler system constructed from elliptic units to interpolate the $L$-values of $\phi^2$ and multiply the resulting distributions with an appropriate twist of the Kubota-Leopoldt $p$-adic $L$-function associated to $\eta$, which interpolates the $L$-values of $\eta$.

Because of the symmetry between the two distributions, we show that some plus and minus logarithms $\log^\pm$ of Pollack divide $\mu_{+\epsilon(p)p^{k-1}}\pm\mu_{-\epsilon(p)p^{k-1}}$. This allows us to obtain two bounded measures:

\begin{theorem}(Theorem~\ref{interpopm}) Let $\theta$ be an even Dirichlet character of conductor $p^n$. There exist bounded $p$-adic measures $\mu^\pm(\symV)$ such that the followings hold.
\begin{itemize}
\item[(a)] If $n$ is even, then
$$\int_{\Zp^\times}\theta\mu^+(\symV)
=\frac{(2k-3)!(k-1)!p^{2n(k-1)}}{\theta\left(\log^+\right)\tau(\theta^{-1})^2\epsilon(p)^n}\times\frac{L(\Sym^2f,\theta^{-1},2k-2)}{\rm (period)};
$$
\item[(b)] If $n$ is odd, then
$$\int_{\Zp^\times}\theta\mu^-(\symV)
=\frac{(2k-3)!(k-1)!p^{2n(k-1)}}{\theta\left(\log^-\right)\tau(\theta^{-1})^2\epsilon(p)^n}\times\frac{L(\Sym^2f,\theta^{-1},2k-2)}{\rm (period)}.
$$
\end{itemize}
Moreover, $\mu^\pm(\symV)$ are uniquely determined by (a) and (b) respectively.
\end{theorem}

In \S~\ref{selmer}, we make use of some of the ideas in \cite{lei09} to show that these measures can be obtained from some appropriate Coleman maps and define the corresponding plus and minus $p$-Selmer groups $\Sel_p^\pm(\symV)$. On identifying the measures as elements in some Iwasawa algebra $\Lambda\otimes\QQ$, we show that the following holds under some appropriate conditions (see Theorem~\ref{MC0} for a precise statement).
\begin{theorem}The Selmer groups $\Sel_p^\pm(\symV)$ are $\Lambda$-cotorsion and
$$\Char_{\Lambda\otimes\QQ}\left(\Sel_p^\pm(\symV)^\vee\right)=\left(\mu^\pm(\symV)\right).$$
\end{theorem}

Finally, in the appendix, we explain how some of the linear algebra results that we use to prove the main theorems can be easily generalised to general symmetric powers $\Sym^mf$ where $m\ge2$ is an integer.


\section{Notation}

\subsection{Extensions by $p$ power roots of unity} 
   Throughout this paper, $p$ is an odd prime. If $K$ is a field of characteristic $0$, either local or global, $G_K$ denotes its absolute Galois group, $\chi$ the $p$-cyclotomic character on $G_K$ and $\calO_K$ the ring of integers of $K$. We write $\iota$ for the complex conjugation in $G_\QQ$.

For an integer $n\ge0$, we write $K_n$ for the extension $K(\mu_{p^n})$ where $\mu_{p^n}$ is the set of $p^n$th roots of unity and $K_\infty$ denotes $\bigcup_{n\ge1} K_n$. When $K=\QQ$, we write $k_n=\QQ(\mu_{p^n})$ instead. In particular, we write $\Qpn=\Qp(\mu_{p^n})$. Let $G_n$ denote the Galois group $\Gal(\Qpn/\Qp)$ for $0\le n \le\infty$. Then, $G_\infty\cong\Delta\times\Gamma$ where $\Delta=G_1$ is a finite group of order $p-1$ and $\Gamma=\Gal(\QQ_{p,\infty}/\QQ_{p,1})\cong\Zp$. We fix a topological generator $\gamma$ of $\Gamma$.


\subsection{Iwasawa algebras and power series}
 Given a finite extension $K$ of $\Qp$, $\Lambda_{\calO_K}(G_\infty)$ (respectively $\Lambda_{\calO_K}(\Gamma)$) denotes the Iwasawa algebra of $G_\infty$ (respectively $\Gamma$) over $\calO_K$. We write $\Lambda_K(G_\infty)=\Lambda_{\calO_K}(G_\infty)\otimes K$ and $\Lambda_K(\Gamma)=\Lambda_{\calO_K}(\Gamma)\otimes K$. If $M$ is a finitely generated $\Lambda_{\calO_K}(\Gamma)$-torsion (respectively $\Lambda_K(\Gamma)$-torsion) module, we write $\Char_{\Lambda_{\calO_K}(\Gamma)}(M)$ (respectively $\Char_{\Lambda_K(\Gamma)}(M)$) for its characteristic ideal.

Given a module $M$ over $\Lambda_{\calO_K}(G_\infty)$ (respectively $\Lambda_K(G_\infty)$) and a character $\delta:\Delta\rightarrow\Zp^\times$, $M^\delta$ denotes the $\delta$-isotypical component of $M$. For any $m\in M$, we write $m^\delta$ for the projection of $m$ into $M^\delta$. The Pontryagin dual of $M$ is written as $M^\vee$.

Let $r\in\RR_{\ge0}$. We define
\[
\HH_r=\left\{\sum_{n\geq0,\sigma\in\Delta}c_{n,\sigma}\cdot\sigma\cdot X^n\in\CC_p[\Delta][[X]]:\sup_{n}\frac{|c_{n,\sigma}|_p}{n^r}<\infty\ \forall\sigma\in\Delta\right\}
\]
where $|\cdot|_p$ is the $p$-adic norm on $\CC_p$ such that $|p|_p=p^{-1}$. We write $\HH_\infty=\cup_{r\ge0}\HH_r$ and $\HH_r(G_\infty)=\{f(\gamma-1):f\in\HH_r\}$
 for $r\in\RR_{\ge0}\cup\{\infty\}$. In other words, the elements of $\HH_r$ (respectively $\HH_r(G_\infty)$) are the power series in $X$ (respectively $\gamma-1$) over $\CC_p[\Delta]$ with growth rate $O(\log_p^r)$. If $F,G\in\HH_\infty$ or $\HH_\infty(G_\infty)$ are such that $F=O(G)$ and $G=O(F)$, we write $F\sim G$.

Given a subfield $K$ of $\CC_p$, we write $\HH_{r,K}=\HH_r\cap K[\Delta][[X]]$ and similarly for $\HH_{r,K}(G_\infty)$. In particular, $\HH_{0,K}(G_\infty)=\Lambda_{K}(G_\infty)$. 

Let $n\in\ZZ$. We define the $K$-linear map $\Tw_n$ from $\HH_{r,K}(G_\infty)$ to itself to be the map that sends $\sigma$ to $\chi(\sigma)^n\sigma$ for all $\sigma\in G_\infty$. It is clearly bijective (with inverse $\Tw_{-n}$).


\subsection{Crystalline representations}
We write $\BB_{\cris}$ and $\BB_{\rm dR}$ for the rings of Fontaine and $\vp$ for the Frobenius acting on these rings. Recall that there exists an element $t\in\BB_{\rm dR}$ such that $\vp(t)=pt$ and $g\cdot t=\chi(g)t$ for $g\in G_{\Qp}$.

Let $V$ be a $p$-adic representation of $G_{\Qp}$.  We denote the Dieudonn\'{e} module by $\DD_{\cris}(V)=(\BB_{\cris}\otimes V)^{G_{\Qp}}$. We say that $V$ is crystalline if $V$ has the same $\Qp$-dimension as $\Dcris(V)$. Fix such a $V$. If $j\in\ZZ$, $\Fil^j\Dcris(V)$ denotes the $j$th de Rham filtration of $\Dcris(V)$.

Let $T$ be a lattice of $V$ which is stable under $G_{\Qp}$. Let $\HIw(T)$ denote the inverse limit $\displaystyle\lim_{\leftarrow}H^1(\Qpn,T)$ with respect to the corestriction and $\HIw(V)=\QQ\otimes\HIw(T)$. Moreover, if $V$ arises from the restriction of a $p$-adic representation of $G_\QQ$ and $T$ is a lattice stable under $G_\QQ$, we write
\[
\Hh^1(T)=\lim_{\stackrel{\longleftarrow}{n}}H^1(\ZZ[\mu_{p^n},1/p],T)\quad\text{and}\quad
\Hh^1(V)=\QQ\otimes\Hh^1(T).
\]
We have localisation maps 
$$\loc:\Hh^1(T)\rightarrow\HIw(T)\quad\text{and}\quad\loc:\Hh^1(V)\rightarrow\HIw(V).$$
If $F$ is a number field, we define the $p$-Selmer group of $T$ over $F$ to by
$$
\Sel_p(T/F)=\ker\left(H^1(K,T\otimes\Qp/\Zp)\rightarrow\prod_v\frac{H^1(F_v,T\otimes\Qp/\Zp)}{H^1_f(F_v,T\otimes\Qp/\Zp)}\right)
$$
where $v$ runs through the places of $F$.

Let $V(j)$ denote the $j$th Tate twist of $V$, i.e. $V(j)=V\otimes\Qp e_j$ where $G_{\Qp}$ acts on $e_j$ via $\chi^j$. We have $$
\Dcris(V(j))=t^{-j}\Dcris(V)\otimes e_j.
$$
For any $v\in\Dcris(V)$, $v_j=v\otimes t^{-j}e_j$ denotes its image in $\Dcris(V(j))$. We write $\Tw_{j}:\HIw(V)\rightarrow\HIw(V(j))$ for the isomorphism defined in \cite[\S~A.4]{perrinriou93}, which depends on a choice of primitive $p$-power roots of unity. 

Finally, we write 
\[
\exp:\Qpn\otimes\Dcris(V)\rightarrow H^1(\Qpn,V)\quad\text{and}\quad\exp^*:H^1(\Qpn,V)\rightarrow\Qpn\otimes\Fil^0\Dcris(V)
\]
for Bloch-Kato's exponential and dual exponential respectively.


\subsection{Imaginary quadratic fields}

Let $K$ be an imaginary quadratic field with ring of integers $\OO$ and idele class group $C_K$. We write $\varepsilon_K$ for the quadratic character associated to $K$, i.e. the character on $G_{\QQ}$ which sends $\sigma$ to $1$ if $\sigma\in G_K$ and to $-1$ otherwise.

A Grossencharacter of $K$ is simply a continuous homomorphism $\phi:C_K\rightarrow\CC^\times$ with complex $L$-function
\[
L(\phi,s)=\prod_v(1-\phi(v)N(v)^{-s})^{-1}
\]
where the product runs through the finite places $v$ of $K$ at which $\phi$ is unramified, $\phi(v)$ is the image of the uniformiser of $K_{v}$ under $\phi$ and $N(v)$ is the norm of $v$. Let $\ff$ be the conductor of $\phi$. We say that $\eta$ is of type $(m,n)$ where $m,n\in\ZZ$ if the restriction of $\eta$ to the archimedean part $\CC^\times$ of $C_K$ is of the form $z\mapsto z^m\bar{z}^n$.

We write $\K=\cup K(p^n\ff)$ where $K(\a)$ denotes the ray class field of $K$ modulo $\a$ if $\a$ is an ideal of $\OO$.

If $T$ is a $\Zp$-representation of $G_K$, we write
$$
\Hh^1_{p^\infty\ff}(T)=\varprojlim_{K'} H^1(\OO_{K'}[1/p],T)\quad\text{and}\quad\Hh^1_{p^\infty\ff}(\QQ\otimes_{\Zp} T)=\Hh^1_{p^\infty\ff}(T)\otimes_{\Zp}\QQ$$
where $K'$ ranges over all finite extensions of $K$ contained in $K(p^\infty\ff)$.

\subsection{Modular forms}\label{modularforms}

   Let $f=\sum a_nq^n$ be a normalised eigen-newform of weight $k\ge2$, level $N$ and character $\epsilon$. We assume that $f$ is a CM modular form, i.e. $L(f,s)=L(\phi,s)$ for some Grossencharacter $\phi$ of an imaginary quadratic field $K$ with conductor $\ff$. Then, $\phi$ is of type $(-k+1,0)$. Moreover, $p$ inerts in $K$ if and only if $f$ is non-ordinary at $p$. In this case, $a_p$ is always $0$. Throughout, we fix such a $p$ with $p\ne2$.

The coefficient field $F_f$ of $f$ is contained in the field of definition of $\phi$. We write $E$ for the completion of this field at a fixed prime above $p$.

We write $V_f$ for the 2-dimensional $E$-linear representation of $G_{\QQ}$ associated to $f$ from \cite{deligne69}, so we have a homomorphism
$$
\rho_f:G_\QQ\rightarrow \text{GL}(V_f).
$$

Throughout the paper, we assume that the following hypothesis holds.
\begin{hypothesis}\label{hyp1}
If $\epsilon$ and $K$ are as above, then $\varepsilon_K\ne\epsilon$.
\end{hypothesis}


\section{$p$-adic $L$-functions}\label{padicl}

\subsection{Grossencharacters over $K$}\label{EU}
We first review some results on Grossencharacters. Let $\eta$ be a Grossencharacter on $G_K$ of conductor $\ff$. We fix a finite extension $E$ of $\Qp$ such that $E$ contains the image of $\eta$. We write $V(\eta)$ for the one-dimensional $E$-linear representation of $G_K$. It is a representation that factors through $\Gal(\K/K)$. For an ideal $\a$ of $\OO$ which is prime to $p\ff$, the Artin symbol $(\a,\K/K)\in\Gal(\K/K)$ acts on $V(\eta)$ as the multiplication by $\eta(\a)^{-1}$. We write $\tilde{\eta}:G_K\rightarrow E^\times$ for the corresponding character.

We write $\tilde{V}_\eta=\IndQK(V(\eta))$. The canonical homomorphism $K\otimes\QQ(\zeta_{p^\infty})\rightarrow K(p^\infty\ff)$ induces a map
$$
\Ind:\Hh^1_{p^\infty\ff}(V(\eta))\rightarrow\Hh^1(\tilde{V}_\eta).
$$
Let $\gamma$ be a non-zero element of $V(\eta)$. 
 By \cite[\S 15.5]{kato04}, a system of norm compatible elliptic units in $K(p^n\ff)$ defines an element $z_{p^\infty\ff}\in\Hh^1_{p^\infty\ff}(\Zp(1))$. We write the image of $z_{p^\infty\ff}$ under the composition
$$\xymatrix@1{\Hh^1_{p^\infty\ff}(\Zp(1))\ar[r]^-{\gamma}&\Hh^1_{p^\infty\ff}(V(\eta)(1))\ar[r]^-{\Ind}&\Hh^1(\tilde{V}_\eta(1))\ar[r]^-{\loc}&\HIw(\tilde{V}_\eta(1))\ar[r]^{\Tw_{-1}}&\HIw(\tilde{V}_\eta)}$$
as $z_\gamma(\eta)=z(\eta)$ and its projection into $H^1(\Qpn,\tilde{V}_\eta(j))$ is denoted by $z_{j,n}(\eta)$.

Note that the eigenvalues of $\iota$ on $\tilde{V}_\eta$ are $\pm1$, each with multiplicity $1$. If $v\in\tilde{V}_\eta$, we write $v^\pm$ for the projection of $v$ into the $\pm1$-eigenspace.

\begin{proposition}\label{interpo}
Let $\eta$ be a Grossencharacter over $K$ of type $(-r,0)$ with $r\ge1$. Let $\theta$ be a character on $G_n$ and write
\begin{eqnarray*}
\kappa_\theta:\Qpn\otimes\Fil^0\Dcris(\tilde{V}_\eta(1))&\rightarrow&\CC\otimes \tilde{V}_\eta(1)\\
x\otimes y&\mapsto&\sum_{\sigma\in G_n}\theta(\sigma)\sigma(x)\per(y)
\end{eqnarray*}
where $\per$ is the period map associated to $\eta$ as defined in \cite[\S15.8]{kato04}. Then, we have
$$
\kappa_\theta\circ\exp^*(z_{1,n}(\eta))=L_{\{p\}}(\bar{\eta}\theta,r)\cdot(\gamma')^\pm
$$
where $\pm=\theta(-1)$ and $\gamma'$ denotes the image of $\gamma$ in $\tilde{V}_\eta$.
\end{proposition}

\begin{proof}
\cite[\S15.12]{kato04}.
\end{proof}


\subsection{The symmetric square of a CM modular form}\label{gosplit}

Let $f$ be a modular form as in \S\ref{modularforms}. By comparing the eigenvalues of Frobenii, we see that the representation $V_f$ is isomorphic to $\tilde{V}_\phi=\IndQK V(\phi)$. Therefore, $V_f$ admits a basis $x,y$ such that for $\sigma\in G_\QQ$, the matrix of $\rho_f(\sigma)$ with respect to this basis is given by
\begin{equation}
\rho_f(\sigma)=
\begin{pmatrix}
\tilde{\phi}(\sigma)&0\\
0&\tilde{\phi}(\iota\sigma\iota)
\end{pmatrix}\label{form1}
\end{equation}
if $\sigma\in G_K$. Otherwise,
\begin{equation}\label{form2}
\rho_f(\sigma)=
\begin{pmatrix}
0&\tilde{\phi}(\iota\sigma'\iota)\\
\tilde{\phi}(\sigma')&0
\end{pmatrix}
\end{equation}
where $\sigma=\iota\sigma'$ with $\sigma'\in G_K$. 

\begin{lemma}\label{det}
The determinant of $\rho_f$ is given by
$$
\det(\rho_f)(\sigma)=
\begin{cases}
\tilde{\phi}(\sigma)\tilde{\phi}(\iota\sigma\iota)&\text{if $\sigma\in G_K$}\\
-\tilde{\phi}(\sigma')\tilde{\phi}(\iota\sigma'\iota)&\text{if $\sigma=\iota\sigma'$ where $\sigma'\in G_K$.}
\end{cases}
$$
\end{lemma}
\begin{proof}
This is immediate from \eqref{form1} and \eqref{form2}.
\end{proof}

\begin{proposition}\label{2reps}
As a $G_\QQ$-representation, $\symV$ decomposes into
\[
\symV\cong V_1\oplus V_2
\]
where $\rho_i:G_\QQ\rightarrow \text{GL}(V_i)$ is an $i$-dimensional representation of $G_\QQ$ for $i=1,2$. Moreover, 
\begin{eqnarray}
\rho_1&\cong& \varepsilon_K\cdot\det(\rho_f)=\varepsilon_K\cdot\epsilon\cdot \chi^{k-1},\label{V1}\\
\rho_2&\cong& \tilde{V}_{\phi^2}.
\end{eqnarray}
\end{proposition}
\begin{proof}
It is clear that $x\otimes x$, $y\otimes y$, $x\otimes y+y\otimes x$ form a basis of $\symV$. By formulae \eqref{form1} and \eqref{form2}, $\sigma\cdot(x\otimes y+y\otimes x)$ is a multiple of $x\otimes y+y\otimes x$ for any $\sigma\in G_{\QQ}$. Hence, it gives an one-dimensional sub-representation $V_1$ of $\symV$. More explicitly, we have
$$
\sigma\cdot(x\otimes y+y\otimes x)=
\begin{cases}
\tilde{\phi}(\sigma)\tilde{\phi}(\iota\sigma\iota)(x\otimes y+y\otimes x)&\text{if $\sigma\in G_K$}\\
\tilde{\phi}(\sigma')\tilde{\phi}(\iota\sigma'\iota)(x\otimes y+y\otimes x)&\text{if $\sigma=\iota\sigma'$ where $\sigma'\in G_K$.}
\end{cases}
$$
Therefore, we deduce \eqref{V1} from Lemma~\ref{det}.

 It is also clear that $x\otimes x$, $y\otimes y$ form a basis of a $2$-dimensional representation $\rho_2:G_\QQ\rightarrow\text{GL}(V_2)$. With respect to this basis, 
\[
\rho_2(\sigma)=
\begin{pmatrix}
\tilde{\phi}^2(\sigma)&0\\
0&\tilde{\phi}^2(\iota\sigma\iota)
\end{pmatrix}
\]
if $\sigma\in G_K$. Otherwise, if $\sigma=\iota\sigma'$ where $\sigma'\in G_K$, then 
\[
\rho_2(\sigma)=
\begin{pmatrix}
0&\tilde{\phi}^2(\iota\sigma'\iota)\\
\tilde{\phi}^2(\sigma')&0
\end{pmatrix}.
\]
Therefore, $V_2\cong \IndQK V(\phi^2)$ as required.
\end{proof}

\begin{corollary}\label{Lfact}
The complex $L$ function admits a factorisation
\[
L(\Sym^2f,s)=L(\phi^2,s)L(\varepsilon_K\cdot\epsilon,s-k+1).
\]
\end{corollary}
\begin{proof}
The $L$-function of $\Sym^2f$ only have non-trivial Euler factors at $q\nmid N$. The Euler factors on the two sides of the equation at $q$ agree by Proposition~\ref{2reps}, so we are done.
\end{proof}


\subsection{The symmetric square as a $G_{\Qp}$-representation}

We study the representation $\symV$ restricted to $G_{\Qp}$. More specifically, we study $\Dcris(\Sym^2V_f)$.

\begin{lemma}
As $G_{\Qp}$-representations, both $V_1$ and $V_2$ are crystalline.
\end{lemma}
\begin{proof}
 The functor $\Dcris$ is compatible with taking direct sums, so we can identify $\Dcris(V_i)$ as a filtered sub-$\vp$-module of $\Dcris(V_f)$ for $i=1,2$. That is,
\begin{equation}\label{split}
\Dcris(\symV)\cong\Dcris(V_1)\oplus\Dcris(V_2).
\end{equation}
Since $\symV$ is crystalline, so $\Dcris(\symV)$ is of dimension $3$ over $E$. Hence, $\Dcris(V_i)$ must have dimension $i$ and $V_i$ is crystalline for $i=1,2$.
\end{proof}

We now give explicit descriptions of $\Dcris(V_1)$ and $\Dcris(V_2)$.

Recall that $\Dcris(V_f)$ is a 2-dimensional $E$-vector space with Hodge-Tate weights $0$ and $1-k$. Moreover, the de Rham filtration is given by
    \begin{equation}
     \Fil^i\Dcris(V_f)=
     \left\{
     \begin{array}{ll}
      E\omega\oplus E\vp(\omega)         & \text{if $i\le0$}\\
      E\omega                     & \text{if $1\le i\le k-1$}\\
      0                          & \text{if $i\ge k$}
     \end{array}\right.
    \end{equation}
for some $\omega\ne0$. The action of $\vp$ on $\Dcris(V_f)$ satisfies $\vp^2=-\epsilon(p)p^{k-1}$. Therefore,

\begin{equation}\label{fil}
     \Fil^i\Dcris(\Sym^2(V_f))=
     \left\{
     \begin{array}{ll}
        \Dcris(\Sym^2(V_f)) & \text{if $i\le 0$} \\
      E(\omega\otimes\omega)\oplus E(\vp(\omega)\otimes\omega+\omega\otimes\vp(\omega))         & \text{if $1\le i\le k-1$}\\
      E(\omega\otimes\omega)                     & \text{if $k\le i\le 2k-2$}\\
      0                          & \text{if $i\ge 2k-1$}
     \end{array}\right.
    \end{equation}

Since $\vp^2(\omega)=-\epsilon(p)p^{k-1}\omega$, we have
\[
\vp\Big(\omega\otimes\vp(\omega)+\vp(\omega)\otimes\omega\Big)=-\epsilon(p)p^{k-1}\Big(\omega\otimes\vp(\omega)+\vp(\omega)\otimes\omega\Big).
\]
In particular, $\omega\otimes\vp(\omega)+\vp(\omega)\otimes\omega$ is an eigenvector of $\vp$. Therefore, we have a decomposition of filtered $\vp$-modules
$$
\Dcris(\Sym^2(V_f))=\Big(E(\omega\otimes\omega)\oplus E(\vp(\omega)\otimes\vp(\omega))\Big)\oplus\Big(E(\omega\otimes\vp(\omega)+\vp(\omega)\otimes\omega)\Big).
$$

\begin{proposition}\label{des}
As filtered $\vp$-modules, we have
\begin{eqnarray*}
\Dcris(V_1)&=&E(\vp(\omega)\otimes\omega+\omega\otimes\vp(\omega)),  \\
\Dcris(V_2)&=&E(\omega\otimes\omega)\oplus E(\vp(\omega)\otimes\vp(\omega)).
\end{eqnarray*}
\end{proposition}
\begin{proof}
By \eqref{V1}, $\rho_1=\varepsilon_K\cdot\epsilon\cdot\chi^{k-1}$. Since $p$ is inert in $K$, $\varepsilon_K(p)=-1$. The Hodge-Tate weight of $V_1$ is therefore $1-k$ and $\vp$ acts on $\Dcris(V_1)$ as multiplication by $-\epsilon(p)p^{k-1}$. This proves the first equality. The second equality is then automatic by \eqref{split}.
\end{proof}
\begin{remark}
Such a decomposition of $G_{\Qp}$-representations is in fact possible for $f$ without CM (see \cite[\S2.2]{perrinriou98}).
\end{remark}

\begin{corollary}\label{eval}
The eigenvalues of $\vp$ on $\Dcris(V_2)$ are $\pm\epsilon(p)p^{k-1}$.
\end{corollary}
\begin{proof}
By Proposition~\ref{des}, the matrix of $\vp$ with respect to the basis $\omega\otimes\omega$, $\vp(\omega)\otimes\vp(\omega)$ is
$$
\begin{pmatrix}
0 & \epsilon(p)^2p^{2k-2}\\
1 & 0
\end{pmatrix},
$$
hence the result.
\end{proof}

\begin{corollary}\label{HT}
The Hodge-Tate weights of $V_2$ are $0$ and $2-2k$.
\end{corollary}
\begin{proof}
This follows from \eqref{fil} and Proposition~\ref{des}.
\end{proof}

\subsection{The Perrin-Riou pairing}

By Corollary~\ref{eval}, the slope of $\vp$ on $\Dcris(V_2)$ is $k-1$. Hence, by Corollary~\ref{HT}, given any $v\in\Dcris(V_2)$, we have the Perrin-Riou pairing
$$
\LL_{v}:\HIw(V_2^*)\rightarrow\HH_{k-1,E}(G_\infty)
$$
which satisfies the following properties.
\begin{proposition}\label{char}
 For an integer $r\ge0$, we have
$$\chi^r\left(\LL_v(\mathbf z)\right)
=r!\left[\left(1-\frac{\vp^{-1}}{p}\right)(1-\vp)^{-1}(v_{r+1}),\exp^*(z_{-r,0})\right]_0.
$$
Let $\theta$ be a character of $G_n$ which does not factor through $G_{n-1}$ with $n\ge1$, then
$$
\chi^r\theta\left(\LL_v(\mathbf z)\right)
=\frac{r!}{\tau(\theta^{-1})}\sum_{\sigma\in G_n}\theta^{-1}(\sigma)\left[\vp^{-n}(v_{r+1}),\exp^*(z_{-r,n}^\sigma)\right]_n
$$
where $[,]_n$ is the pairing
\[
[,]_n:H^1(\Qpn,V_2(r+1))\times H^1(\Qpn,V_2^*(-r))\rightarrow H^2(\Qpn,E(1))\cong E,
\]
$z_{-r,n}$ denotes the projection of $\Tw_{-r}(z)$ into $H^1(\Qpn,V_2^*(-r))$ and $\tau(\theta^{-1})$ denotes the Gauss sum of $\theta^{-1}$.
\end{proposition}
\begin{proof}
See \cite[\S3.2]{lei09}.
\end{proof}

\begin{remark}
The assumption on the eigenvalues of $\vp$ made in \cite{lei09} are not necessary for our purposes here because the Perrin-Riou pairings can be defined by applying $1-\vp$ to the $(\vp,G_\infty)$-module of $V_2^*$ (see \cite{leiloefflerzerbes10} and \cite[\S16.4]{kato04}). 
\end{remark}

We fix a non-zero element $\bar{\omega}\in\Fil^{-1}\Dcris(V_2^*(1))$ and write
\[
\per(\bar{\omega})=\Omega_+(\gamma')^++\Omega_-(\gamma')^-
\]
where $\Omega_\pm\in\CC^\times$ and $\gamma'$ is as in the statement of Proposition~\ref{interpo} for some fixed $\gamma$.

\begin{definition}
Under the choices made above, we define $v^\pm\in\Dcris(V_2)$ by 
$$
v^\pm=\frac{1}{[\vp(\omega)\otimes\vp(\omega),\bar{\omega}]}\Big(\pm\epsilon(p)p^{k-1}\omega\otimes\omega+\vp(\omega)\otimes\vp(\omega)\Big).
$$
\end{definition}

\begin{lemma}\label{properties}
The elements $v^\pm$ satisfy:
\begin{itemize}
\item[(a)] Both $v^{\pm}$ are eigenvalues of $\vp$ with $\vp(v^\pm)=\pm\epsilon(p)p^{k-1}v^\pm$;
\item[(b)] For any $x\in\Fil^0\Dcris(V_2^*(-r))$ and an integer $r$ such that $0\le r\le 2k-3$, we have
$$
[v_{r+1}^+,x]=[v_{r+1}^-,x]
$$
where $[,]$ denotes the pairing
$$
[,]:\Dcris(V_2(r+1))\times\Dcris(V_2^*(-r))\rightarrow \Dcris(E(1))=E\cdot t^{-1}e_1.
$$
\end{itemize}
\end{lemma}
\begin{proof}
(a) is easy to check using the matrix given in the proof of Corollary~\ref{eval} (or by direct calculations).

By Corollary~\ref{HT}, the Hodge-Tate weights of $V_2^*$ are $0$ and $2k-2$. Hence, $\Fil^0\Dcris(V_2^*(-r))$ is one-dimensional with basis $\bar{\omega}_{-r-1}$ for $0\le r \le 2k-3$. Since $(\omega\otimes\omega)_{r+1}\in\Fil^0\Dcris(V_2(r+1))$, we have $[(\omega\otimes\omega)_{r+1},\bar{\omega}_{-r-1}]=0$. Hence, 
$$[v_{r+1}^+,\bar{\omega}_{-r-1}]=[v_{r+1}^-,\bar{\omega}_{-r-1}]=1,$$  
which implies (b).
\end{proof}

Note that $V_2^*\cong \tilde{V}_{\bar{\phi}^2}(2k-2)$. This enables us to make the following definition of $p$-adic $L$-functions associated to $\phi^2$.
\begin{definition}
On taking $\eta=\bar{\phi}^2$ in \S\ref{EU}, we define
$$
L_{\pm\epsilon(p)p^{k-1}}(\phi^2)=\LL_{v^\pm}\left(\Tw_{2k-2}\left(z(\bar{\phi}^2)\right)\right)\in\HH_{k-1,E}(G_\infty).
$$
\end{definition}

\begin{lemma}\label{twistinterpo}
Let $\theta$ be a character of $G_n$ which does not factor through $G_{n-1}$ with $n\ge1$ and write $\delta=\theta(-1)$, then
\[
\chi^{2k-3}\theta\Big(L_{\alpha}(\phi^2)\Big)=\frac{(2k-3)!p^{(2k-2)n}}{\tau(\theta^{-1})\alpha^{n}}\times\frac{L(\phi^2\theta^{-1},2k-2)}{\Omega_\delta}
\]
where $\alpha=\pm\epsilon(p)p^{k-1}$.
\end{lemma}
\begin{proof}
We have
\begin{eqnarray*}
&&\chi^{2k-3}\theta\Big(L_{\pm\epsilon(p)p^{k-1}}(\phi^2)\Big)\\
&=&\chi^{2k-3}\theta\Big(\LL_{v^\pm}\left(\Tw_{2k-2}\left(z(\bar{\phi}^2)\right)\right)\Big)\\
&=&\frac{(2k-3)!}{\tau(\theta^{-1})}\sum_{\sigma\in G_n}\theta^{-1}(\sigma)\left[\vp^{-n}(v^\pm_{2k-2}),\exp^*(z_{1,n}(\bar{\phi}^2)^\sigma)\right]_n\\
&=&\frac{(2k-3)!}{\tau(\theta^{-1})}\left[\left(\pm\epsilon(p)p^{k-1}\times p^{-2k+2}\right)^{-n}v^\pm_{2k-2},\sum_{\sigma\in G_n}\theta^{-1}(\sigma)\exp^*(z_{1,n}(\bar{\phi}^2)^\sigma)\right]_n\\
&=&\frac{(2k-3)!p^{(2k-2)n}}{\tau(\theta^{-1})\left(\pm\epsilon(p)p^{k-1}\right)^{n}}\times\frac{L(\phi^2\theta^{-1},2k-2)}{\Omega_\delta}
\end{eqnarray*}
where the second equality follows from Proposition~\ref{char}, the third follows from Lemma~\ref{properties}(a) and the last equality is a consequence of Proposition~\ref{interpo} and the fact that $p$ divides the conductor of $\theta$.
\end{proof}

\begin{lemma}\label{trivialinterpo}We have
\[
\chi^{2k-3}\Big(L_{\pm\epsilon(p)p^{k-1}}(\phi^2)\Big)=\Big(1-p^{-1}+(1-\epsilon(p)^{-2}p^{2k-3})(\pm\epsilon(p)p^{1-k})\Big)\times\frac{L(\phi^2,2k-2)}{\Omega_+}.
\]
\end{lemma}

\begin{proof}
Since $\vp^2=\epsilon(p)^2p^{2-2k}$ on $\Dcris(V_2(2k-2))$, we have
\begin{eqnarray*}
&&\left(1-\frac{\vp^{-1}}{p}\right)(1-\vp)^{-1}\\
&=&(1-\epsilon(p)^{-2}p^{2k-3}\vp)\frac{1+\vp}{1-\epsilon(p)^2p^{2-2k}}\\
&=&\frac{1-p^{-1}+(1-\epsilon(p)^{-2}p^{2k-3})\vp}{1-\epsilon(p)^2p^{2-2k}}.
\end{eqnarray*}
Therefore, similarly to the proof of Lemma~\ref{twistinterpo}, we have
\begin{eqnarray*}
&&\chi^{2k-3}\Big(L_{\pm\epsilon(p)p^{k-1}}(\phi^2)\Big)\\
&=&\chi^{2k-3}\Big(\LL_{v^\pm}\left(\Tw_{2k-2}\left(z(\bar{\phi}^2)\right)\right)\Big)\\
&=&(2k-3)!\left[\frac{1-p^{-1}+(1-\epsilon(p)^{-2}p^{2k-3})\vp}{1-\epsilon(p)^2p^{2-2k}}\left(v^\pm_{2k-2}\right),\exp^*(z_{1,0}(\bar{\phi}^2))\right]_0\\
&=&(2k-3)!\left[\frac{1-p^{-1}+(1-\epsilon(p)^{-2}p^{2k-3})(\pm\epsilon(p)p^{1-k})}{1-\epsilon(p)^2p^{2-2k}}\cdot v^\pm_{2k-2},\exp^*(z_{1,0}(\bar{\phi}^2))\right]_0\\
&=&\frac{1-p^{-1}+(1-\epsilon(p)^{-2}p^{2k-3})(\pm\epsilon(p)p^{1-k})}{1-\epsilon(p)^2p^{2-2k}}\times\frac{L_{\{p\}}(\phi^2,2k-2)}{\Omega_+}\\
&=&\Big(1-p^{-1}+(1-\epsilon(p)^{-2}p^{2k-3})(\pm\epsilon(p)p^{1-k})\Big)\times\frac{L(\phi^2,2k-2)}{\Omega_+}.
\end{eqnarray*}
\end{proof}

\begin{remark}
Consider the $p$-adic $L$-function $L_{+\epsilon(p)p^{k-1}}(\phi^2)$. The first factor on the right-hand side of the equation in the statement of Lemma~\ref{trivialinterpo} vanishes if and only if $k=2$ and $\epsilon(p)=1$ (e.g. when $f$ corresponds to an elliptic curve over $\QQ$). This recovers the trivial zero result in \cite{perrinriou98}.
\end{remark}


\subsection{$p$-adic $L$-functions of the symmetric square}

Let us first recall the following result of Kubota and Leopoldt.

\begin{theorem}\label{KL}
If $\eta$ is a non-trivial Dirichlet character of conductor prime to $p$, there exists a bounded $p$-adic measure $L_p(\eta)\in\HH_{0,E}(G_\infty)$ where $E$ is some finite extension of $\Qp$ which contains the image of $\eta$ such that 
\begin{eqnarray*}
\chi^r\theta(L_p(\eta))&=&\frac{(r+1)!p^{n(r+1)}}{(2\pi i)^{r+1}\tau(\theta^{-1})}\times L(\eta\theta^{-1},r+1);\\
\chi^r(L_p(\eta))&=&\frac{(r+1)!}{(2\pi i)^{r+1}} L(\eta,r+1).
\end{eqnarray*}
for any integer $r\ge0$ and Dirichlet character $\theta$ of conductor $p^n$ such that $\chi^{r+1}\theta(-1)=\eta(-1)$.
\end{theorem}

Since we assume that Hypothesis~\ref{hyp1} holds, we may take $\eta=\varepsilon_K\cdot\epsilon$ in Theorem~\ref{KL}. This enables us to give the following definition.

\begin{definition}
For $\alpha=\pm\epsilon(p)p^{k-1}$ we define
$$
L_{\alpha}\left(\symV\right)=L_{\alpha}(\phi^2)\times\Tw_{-k+1}\left( L_p(\varepsilon_K\cdot\epsilon)\right).
$$
\end{definition}

For the rest of this section, unless otherwise stated, $\theta$ denotes an even character on $G_n$ which does not factor through $G_{n-1}$ with $n\ge1$.

\begin{theorem}\label{prelog}
Both $L_{\pm\epsilon(p)p^{k-1}}\left(\symV\right)$ lie inside $\HH_{k-1,E}(G_\infty)$ and admit the following interpolating properties:
\begin{eqnarray*}
\chi^{2k-3}\theta\Big(L_{\alpha}\left(\symV\right)\Big)&=&\frac{(2k-3)!(k-1)!p^{3n(k-1)}}{\tau(\theta^{-1})^2\alpha^{n}}\times\frac{L(\Sym^2f,\theta^{-1},2k-2)}{(2\pi i)^{k-1}\Omega_+};\\
\chi^{2k-3}\Big(L_{\alpha}\left(\symV\right)\Big)&=&(2k-3)!(k-1)!\left(1-\frac{1}{p}+\alpha\left(p^{-2k+2}-\frac{1}{p\epsilon(p)^{2}}\right)\right)\times\frac{L(\Sym^2f,2k-2)}{(2\pi i)^{k-1}\Omega_+}
\end{eqnarray*}
where $\alpha=\pm\epsilon(p)p^{k-1}$.
\end{theorem}
\begin{proof}
By definition, $L_{\alpha}(\phi^2)\in\HH_{k-1,E}(G_\infty)$ and $L_p(\varepsilon_K\cdot\epsilon)\in\HH_{0,E}(G_\infty)$ which implies the first part of the theorem.

Since $\det(V_f)=\epsilon\chi^{k-1}$ and $\rho_f$ is odd, we have $\epsilon\chi^{k-1}(-1)=-1$. But $\varepsilon_K(-1)=-1$ and $\theta(-1)=1$, so $\chi^{k-1}\theta(-1)=\varepsilon_K\epsilon(-1)$ and we can apply Theorem~\ref{KL} and Lemma~\ref{twistinterpo} as follows:
\begin{eqnarray*}
&&\chi^{2k-3}\theta\Big(L_{\alpha}\left(\symV\right)\Big)\\
&=&\chi^{2k-3}\theta\Big(L_{\alpha}(\phi^2)\Big)\times\chi^{k-2}\theta\Big( L_p(\varepsilon_K\cdot\epsilon)\Big)\\
&=&\frac{(2k-3)!p^{(2k-2)n}}{\tau(\theta^{-1})\alpha^{n}}\times\frac{L(\phi^2\theta,2k-2)}{\Omega_+}\times\frac{(k-1)!p^{n(k-1)}}{(2\pi i)^{k-1}\tau(\theta^{-1})}\times L(\varepsilon_k\cdot\epsilon\cdot\theta^{-1},k-1)\\
&=&\frac{(2k-3)!(k-1)!p^{3n(k-1)}}{\tau(\theta^{-1})^2\alpha^{n}}\times\frac{L(\Sym^2f,\theta^{-1},2k-2)}{(2\pi i)^{k-1}\Omega_+},
\end{eqnarray*}
where the last equality follows from Corollary~\ref{Lfact}. This gives the first interpolating formula and the second one can be deduced in the same way.
\end{proof}

\begin{lemma}\label{nonzero}
Let $\eta$ be an even character on $\Delta$, then $L_{\pm\epsilon(p)p^{k-1}}^\eta\left(\symV\right)\ne0$.
\end{lemma}
\begin{proof}
We have $L(\symV,\eta,2k-2)\ne0$ because the critical strip of $\symV$ is $k-1<\text{Re}(s)<k$. Therefore, we are done by the interpolating properties given by Theorem~\ref{prelog}.
\end{proof}
\subsection{Pollack's plus and minus splittings}

As in \cite{pollack03}, we define
\begin{eqnarray*}
\log^+(\gamma)&=&\prod_{r=0}^{2k-3}\prod_{n=1}^\infty\frac{\Phi_{2n}(\chi(\gamma)^{-r}\gamma)}{p},\\
\log^-(\gamma)&=&\prod_{r=0}^{2k-3}\prod_{n=1}^\infty\frac{\Phi_{2n-1}(\chi(\gamma)^{-r}\gamma)}{p},
\end{eqnarray*}
where $\Phi_m$ denotes the $p^m$th cyclotomic polynomial. Then, $\log^\pm(\gamma)\sim\log^{k-1}$.

\begin{lemma}\label{zeros}
For an integer $r$ such that $0\le r \le 2k-3$ and a character $\theta$ of $G_n$ which does not factor through $G_{n-1}$ with $n\ge1$, then
$$
\chi^{r}\theta\Big(L_{+\epsilon(p)p^{k-1}}(\phi^2)\Big)=(-1)^n\chi^{r}\theta\Big(L_{-\epsilon(p)p^{k-1}}(\phi^2)\Big).
$$
\end{lemma}
\begin{proof}
This follows from the same calculations as in the proof of Lemma~\ref{twistinterpo} thanks to Lemma~\ref{properties}(b).
\end{proof}

\begin{corollary}
We have divisibilities
\begin{eqnarray*}
\log^+(\gamma)&|&L_{+\epsilon(p)p^{k-1}}(\phi^2)+L_{-\epsilon(p)p^{k-1}}(\phi^2);\\
\log^-(\gamma)&|&L_{+\epsilon(p)p^{k-1}}(\phi^2)-L_{-\epsilon(p)p^{k-1}}(\phi^2).
\end{eqnarray*}
Similarly,
\begin{eqnarray*}
\log^+(\gamma)&|&L_{+\epsilon(p)p^{k-1}}\left(\symV\right)+L_{-\epsilon(p)p^{k-1}}\left(\symV\right);\\
\log^-(\gamma)&|&L_{+\epsilon(p)p^{k-1}}\left(\symV\right)-L_{-\epsilon(p)p^{k-1}}\left(\symV\right).
\end{eqnarray*}
\end{corollary}
\begin{proof}
The first set of divisibilities follows from Lemma~\ref{zeros}. The second set is then immediate by definition.
\end{proof}
This allows us to define the following.

\begin{definition}We define the plus and minus $p$-adic $L$-functions for $\symV$ by
\begin{eqnarray*}
L_p^+(\symV)&=&\Big(L_{+\epsilon(p)p^{k-1}}\left(\symV\right)+L_{-\epsilon(p)p^{k-1}}\left(\symV\right)\Big)/2\log^+(\gamma);\\
L_p^-(\symV)&=&\Big(L_{+\epsilon(p)p^{k-1}}\left(\symV\right)-L_{-\epsilon(p)p^{k-1}}\left(\symV\right)\Big)/2\log^-(\gamma).
\end{eqnarray*}
Similarly, we define the plus and minus $p$-adic $L$-functions for $V_2$ by
\begin{eqnarray*}
L_p^+(\phi^2)&=&\Big(L_{+\epsilon(p)p^{k-1}}\left(\phi^2\right)+L_{-\epsilon(p)p^{k-1}}\left(\phi^2\right)\Big)/2\log^+(\gamma);\\
L_p^-(\phi^2)&=&\Big(L_{+\epsilon(p)p^{k-1}}\left(\phi^2\right)-L_{-\epsilon(p)p^{k-1}}\left(\phi^2\right)\Big)/2\log^-(\gamma).
\end{eqnarray*}
\end{definition}

It is immediate that we have
\begin{equation}\label{facts}
L_p^\pm\left(\symV\right)=L_p^\pm(\phi^2)\times\Tw_{-k+1}\left( L_p(\varepsilon_K\cdot\epsilon)\right).
\end{equation}

\begin{theorem}\label{interpopm}
Both $L_p^\pm(\symV)$ are elements of $\Lambda_E(G_\infty)$ and admit the following interpolating properties:
\begin{itemize}
\item[(a)] If $n$ is even, then
\begin{eqnarray*}
\chi^{2k-3}\theta\Big(L_p^+(\symV)\Big)&=&\frac{(2k-3)!(k-1)!p^{2n(k-1)}}{\log^+\left(\chi^{2k-3}\theta(\gamma)\right)\tau(\theta^{-1})^2\epsilon(p)^n}\times\frac{L(\Sym^2f,\theta^{-1},2k-2)}{(2\pi i)^{k-1}\Omega_+},\\
\chi^{2k-3}\Big(L_p^+(\symV)\Big)&=&\frac{(2k-3)!(k-1)!\left(1-p^{-1}\right)}{\log^+\left(\chi^{2k-3}(\gamma)\right)}\times\frac{L(\Sym^2f,2k-2)}{(2\pi i)^{k-1}\Omega_+};
\end{eqnarray*}
\item[(b)] If $n$ is odd, then
\begin{eqnarray*}
\chi^{2k-3}\theta\Big(L_p^-(\symV)\Big)&=&\frac{(2k-3)!(k-1)!p^{2n(k-1)}}{\log^-\left(\chi^{2k-3}\theta(\gamma)\right)\tau(\theta^{-1})^2\epsilon(p)^n}\times\frac{L(\Sym^2f,\theta^{-1},2k-2)}{(2\pi i)^{k-1}\Omega_+},\\
\chi^{2k-3}\Big(L_p^-(\symV)\Big)&=&\frac{(2k-3)!(k-1)!\left(\epsilon(p)p^{-k+1}-\epsilon(p)^{-1}p^{k-2}\right)}{\log^-\left(\chi^{2k-3}(\gamma)\right)}\times\frac{L(\Sym^2f,2k-2)}{(2\pi i)^{k-1}\Omega_+}.
\end{eqnarray*}
\end{itemize}

Moreover, $L_p^\pm(\symV)$ are uniquely determined by (a) and (b) respectively.
\end{theorem}
\begin{proof}
By the first part of Theorem~\ref{prelog}, $L_{\pm\epsilon(p)p^{k-1}}\left(\symV\right)$ are both elements of $\HH_{k-1,E}(G_\infty)$. But $\log^\pm(\gamma)\sim \log^{k-1}$, so the quotients above are in $\HH_{0,E}(G_\infty)=\Lambda_E(G_\infty)$. 

The interpolating formulae in (a) and (b) follow from those given in Theorem~\ref{prelog}.

Finally, since $L_p^\pm(\symV)\in\Lambda_E(G_\infty)$, they are uniquely determined by their values at an infinite number of characters, hence the last part of the theorem.
\end{proof}

\begin{lemma}\label{not0}
Let $\eta$ be an even character on $\Delta$, then $L_p^{\pm,\eta}\left(\symV\right)\ne0$.
\end{lemma}
\begin{proof}
The same as the proof of Lemma~\ref{nonzero}.
\end{proof}

\begin{remark}
Analogues of Theorem~\ref{interpopm} and Lemma~\ref{not0} for $L_p^\pm(\phi^2)$ can be deduced in the same way.
\end{remark}

\begin{remark}
A conjectural generalisation of Pollack's plus and minus splittings of $p$-adic $L$-functions for motives has been formulated in \cite{dabrowski11}. Theorem~\ref{interpopm} gives an affirmative answer to Conjecture~2 of op. cit. for the special case when the motive corresponds to the symmetric square of a CM modular form.
\end{remark}


\section{Selmer groups}\label{selmer}
In this section, we define the plus and minus $p$-Selmer groups for $\symV$ and relate them to the $p$-adic $L$-functions $L_p^\pm\left(\symV\right)$ defined above. By the decomposition given by Proposition~\ref{2reps}, we only need to define their counterparts for $V_2=\tilde{V}_{\phi^2}$ because the Selmer group of $V_1$ is relatively well-understood. The $G_{\QQ}$-representation $V_2$ behaves in exactly the same way as $V_{f'}$ where $f'$ is some CM  modular form of weight $2k-1$, so many of the results on $V_2$ below can be proved using the arguments given in \cite{lei09}. Therefore, we only outline the proofs without giving all the details here. 

\subsection{Coleman maps and Selmer groups}

As in \cite{lei09,leiloefflerzerbes10}, we define plus and minus Selmer groups using the kernels of some Coleman maps.

\begin{proposition}
If $z\in\HIw(V_2^*)$, then
\begin{eqnarray*}
\log^+(\gamma)&|&\LL_{\vp(\omega)\otimes\vp(\omega)}(z),\\
\log^-(\gamma)&|&\LL_{\omega\otimes\omega}(z).
\end{eqnarray*}
\end{proposition}
\begin{proof}
As in \cite[Proposition~3.14]{lei09}, this can be proved using Proposition~\ref{char}.
\end{proof}

Therefore, as in \cite{lei09}, we may define $\Lambda_E(G_\infty)$-homomorphisms
\begin{eqnarray*}
\col^+:\HIw(V_2^*)&\rightarrow&\Lambda_E(G_\infty)\\
z&\mapsto&\frac{1}{2[\vp(\omega)\otimes\vp(\omega),\bar{\omega}]\log^+(\gamma)}\LL_{\vp(\omega)\otimes\vp(\omega)}(z);\\
\col^-:\HIw(V_2^*)&\rightarrow&\Lambda_E(G_\infty)\\
z&\mapsto&\frac{1}{2[\vp(\omega)\otimes\vp(\omega),\bar{\omega}]\log^-(\gamma)}\LL_{\omega\otimes\omega}(z).
\end{eqnarray*}
Then, it is clear by definition that $\col^\pm\left(\Tw_{2k-2}\left(z(\bar{\phi}^2)\right)\right)=L_p^\pm(\phi^2)$.

We now fix an $\calO_E$-lattice $T$ of $V(\phi)$ which is stable under $G_\QQ$, it then gives rise to natural $\calO_E$-lattices $T_f=\IndQK(T)$ and $\Sym^2T_f$ in $V_f=\tilde{V}_\phi$ and $\symV$ respectively, both of which are again stable under $G_\QQ$. As $p\ne2$, we have
$$
\Sym^2T_f\cong T_1\oplus T_2\quad\text{and}\quad\Sym^2V_f/T_f\cong V_1/T_1\oplus V_2/T_2
$$ 
for some $\calO_E$-lattice $T_i$ inside $V_i$ for $i=1,2$. 

Write $H^1_\pm(\Qpn,T_2^*)$ for the projection of $\ker(\col^\pm)$ into $H^1(\Qpn,T_2^*)$ and define $H^1(\Qpn,V_2/T_2(1))^\pm$
to be the exact annihilator of $H^1_\pm(\Qpn,T_2^*)$ under the Pontryagin duality
$$
H^1(\Qpn,T_2^*)\times H^1(\Qpn,V_2/T_2(1))\rightarrow\Qp/\Zp.
$$

Let $F$ be a number field. Then the $p$-Selmer group of $\Sym^2T_f(1)$ decomposes into those of $T_1(1)$ and $T_2(1)$: 
$$
\Sel_p(\Sym^2T_f(1)/F)=\Sel_p(T_1(1)/F)\oplus \Sel_p(T_2(1)/F).
$$
We define the plus/minus Selmer groups over $k_n=\QQ(\mu_{p^n})$ by
\begin{eqnarray*}
\Sel_p^\pm(T_2(1)/k_n)&=&\ker\left(\Sel_p(T_2(1)/k_n)\rightarrow\frac{H^1(\Qpn,V_2/T_2(1))}{H^1_f(\Qpn,V_2/T_2(1))^\pm}\right),\\
\Sel_p^\pm(\Sym^2T_f(1)/k_n)&=&\Sel_p(T_1(1)/k_n)\oplus \Sel_p^\pm(T_2(1)/k_n)
\end{eqnarray*}
and let
$$
\Sel_p^\pm(T_2(1)/k_\infty)=\lim_{\rightarrow}\Sel_p^\pm(T_2(1)/k_n)\quad\text{and}\quad\Sel_p^\pm(\Sym^2T_f(1)/k_\infty)=\lim_{\rightarrow}\Sel_p^\pm(\Sym^2T_f(1)/k_n).
$$


\subsection{Description of the kernels}
In this section, we give a more explicit description of the groups $H^1_f(\Qpn,V_2/T_2(1))^\pm$ under the following additional assumption.
\begin{hypothesis}\label{hyp2}
Either $p-1\nmid k-1$ or $\epsilon\ne1$.
\end{hypothesis}

In \cite[\S4]{lei09}, one of the key ingredients to give an explicit description of $H^1_f(\Qpn,V_f/T_f(1))^\pm$ is the fact that $\left(V_f/T_f(j)\right)^{G_{\Qpn}}=0$ under some appropriate assumptions. We show below that we get an analogue of such description under Hypothesis~\ref{hyp2}.

\begin{lemma}\label{nofix}
If Hypothesis~\ref{hyp2} holds, then $\left(V_2/T_2(j)\right)^{G_{\Qpn}}=0$ for all $j\in\ZZ$ and $n\in\ZZ_{\ge0}$.
\end{lemma}
\begin{proof}
Let $q\nmid N$ be a prime which is inert in $K$. Then, by the second half of the proof of Proposition~\ref{2reps}, we see that the eigenvalues of the $q$-Frobenius on $V_2(j)$ are $\pm\epsilon(q)\chi^j(q)q^{k-1}$. Therefore, as in \cite[proof of Lemma~4.4]{lei09}, it is enough to show that there exists some $q$ such that
$$\pm\epsilon(q)\chi^j(q)q^{k-1}\not\equiv1\mod p.$$
If either $p-1\nmid k-1$ or $\epsilon(q)\ne1$, we can find such a $q$ by Dirichlet's theorem, so we are done.
\end{proof}

\begin{corollary}
If Hypothesis~\ref{hyp2} holds, then the restriction map $H^1(\QQ_{p,m},T_2(1))\rightarrow H^1(\Qpn,T_2(1))$ is injective for any integers $n\ge m\ge 0$. On identifying the former as a subgroup of the latter, we have
$$
H^1_f(\Qpn,V_2/T_2(1))^\pm=H^1_f(\Qpn,T_2(1))^\pm\otimes E/\calO_E.
$$
Here
$$
H^1_f(\Qpn,T_2(1))^\pm=\left\{x\in H^1_f(\Qpn,T_2(1)):\cor_{n/m+1}(x)\in H^1_f(\QQ_{p,m},T_2(1))\forall m\in S_n^\pm \right\}
$$
where $\cor$ denotes the corestriction map and
\begin{eqnarray*}
S_n^+&=&\{m\in[0,m-1]:m\text{ even}\},\\
S_n^-&=&\{m\in[0,m-1]:m\text{ odd}\}.
\end{eqnarray*}
\end{corollary}
\begin{proof}
These can be proved in exactly the same way as their counterparts in \cite[\S4]{lei09} using Lemma~\ref{nofix}.
\end{proof}

\subsection{Main conjectures}

\begin{theorem}\label{MC}
Let $\theta$ be a character on $\Delta$  and $r\ge0$ an integer such that $\chi^{r+1}\theta(-1)=\eta(-1)$. Then $\Sel_p\left(\Zp(\eta)(r+1)\right)^\theta$ is $\Lambda_E(\Gamma)$-cotorsion and
$$
\Char_{\Lambda_E(\Gamma)}\left(\Sel_p\left(\Zp(\eta)(r+1)\right)^{\vee,\theta}\right)=\left(\Tw_{-r}L_p^\theta(\eta)\right).
$$
\end{theorem}
\begin{proof}
For any $\Lambda_E(G_\infty)$-module, $M^\vee(r)=M(-r)^\vee$. If $M$ is a $\Lambda_E(\Gamma)$-torsion module, we have $\Char(M(r))=\Tw_r(\Char(M))$. Therefore, the result is just a rewrite of the Iwasawa main conjecture, as proved by Mazur-Wiles \cite{mazurwiles}.
\end{proof}

\begin{corollary}\label{MC1}
Let $\eta$ be an even character on $\Delta$. Then
$$
\Char_{\Lambda_E(\Gamma)}\left(\Sel_p(T_1(1)/k_\infty)^{\vee,\eta}\right)=(\Tw_{-k+1}L_p^\eta(\varepsilon_K\cdot\epsilon)).
$$
\end{corollary}
\begin{proof}
We may apply Theorem~\ref{MC} to $\varepsilon_K\cdot\epsilon$ with $r=k-1$.
\end{proof}

\begin{proposition}\label{MC2}
Let $\delta=\pm$ and let $\eta$ be a character on $\Delta$ such that $\eta=1$ if $\delta=-$. Then, $\Sel_p^\delta(T_2(1)/k_\infty)^\theta$ is $\Lambda_E(\Gamma)$-cotorsion and
$$
\Char_{\Lambda_E(\Gamma)}\left(\Sel_p^\delta(T_2(1)/k_\infty)^{\vee,\eta}\right)=\left(L_p^{\delta,\eta}(\phi^2)\right).
$$
\end{proposition}
\begin{proof}
This follows from the same argument as in \cite{pollackrubin04}, which has been generalised for CM modular forms in \cite[\S7]{lei09}. It relies on the main conjecture for $K$ as proved in \cite{rubin91}.
\end{proof}
\begin{theorem}\label{MC0}
Let $\eta$ be character on $\Delta$ as in the statement of Proposition~\ref{MC2}. Then $\Sel_p^\pm(\symV/k_\infty)^{\eta}$ is $\Lambda_E(\Gamma)$-cotorsion and
$$
\Char_{\Lambda_E(\Gamma)}\left(\Sel_p^\pm(\symV/k_\infty)^{\vee,\eta}\right)=\left(L_p^{\pm,\eta}(\symV)\right).
$$
\end{theorem}
\begin{proof}
Recall that 
\[
\Sel_p^\pm(\Sym^2T_f(1)/k_\infty)=\Sel_p(T_1(1)/k_\infty)\oplus \Sel_p^\pm(T_2(1)/k_\infty)
\]
by definition, so
\[
\Sel_p^\pm(\Sym^2T_f(1)/k_\infty)^{\vee,\eta}=\Sel_p(T_1(1)/k_\infty)^{\vee,\eta}\oplus \Sel_p^\pm(T_2(1)/k_\infty)^{\vee,\eta}.
\]
But we have
$$
L_p^{\pm,\eta}\left(\symV\right)=L_p^{\pm,\eta}(\phi^2)\times\Tw_{-k+1}\left( L_p^\eta(\varepsilon_K\cdot\epsilon)\right)
$$
by \eqref{facts}. Therefore, the theorem follows from Corollary~\ref{MC1} and Proposition~\ref{MC2} because $$\Char(M_1\oplus M_2)=\Char(M_1)\Char(M_2)$$ for any torsion modules $M_1$ and $M_2$.
\end{proof}

\section{Appendix}
In this section, we fix an integer $m\ge2$. We prove an analogue of Proposition~\ref{2reps}.
\begin{proposition}
If $m$ is even, we have a decomposition of $G_\QQ$-representations
$$
\Sym^mV_f\cong\bigoplus_{i=0}^{m/2-1} \Big(\tilde{V}_{\phi^{m-2i}}\otimes\left(\varepsilon_K\det \rho_f\right)^i\Big)\oplus\left(\varepsilon_K\det \rho_f\right)^{m/2}.$$
If $m$ is odd, then
$$
\Sym^mV_f\cong\bigoplus_{i=0}^{(m-1)/2} \Big(\tilde{V}_{\phi^{m-2i}}\otimes\left(\varepsilon_K\det \rho_f\right)^i\Big).
$$
\end{proposition}
\begin{proof}
We only give the proof for the case when $m$ is even since the other case can be proved in a similar way. Let $x,y$ be the basis of $V_f$ given as in \S\ref{gosplit}. For an integer $r$ such that $0\le r\le m$, we write $x_r$ for the element in $V_f^{\otimes m}$ given by 
$$
\sum a_1\otimes a_2\otimes\cdots \otimes a_m
$$
where the sum runs over $a_i\in\{x,y\}$ with $\#\{i:a_i=x\}=r$. Then, $x_0,\ldots,x_m$ give a basis of $\Sym^mV_f$.

If $\sigma\in G_K$, we have
$$
\sigma(x_r)=\tilde{\phi}^r(\sigma)\tilde{\phi}^{m-r}(\iota\sigma\iota)x_r
$$
by \eqref{form1}. If $\sigma=\iota\sigma'$ with $\sigma'\in G_K$, then
$$
\sigma(x_r)=\tilde{\phi}^r(\sigma')\tilde{\phi}^{m-r}(\iota\sigma'\iota)x_{m-r}
$$
by \eqref{form2}. Therefore, $x_r$ and $x_{m-r}$ generate a subrepresentation of $\Sym^mV_f$, which we denote by $\rho_r:G_\QQ\rightarrow {\rm GL}(V_r)$ where $0\le r \le m/2$. Note that $V_r$ is 2-dimensional if $r<m/2$ and $V_{m/2}$ is 1-dimensional. We have a decomposition
$$
\Sym^mV_f\cong\oplus_{r=0}^{m/2}V_{r}.
$$

For $r< m/2$, the matrix of $\sigma\in G_K$ respect to the basis $x_{m-r},x_{r}$ is
$$
\begin{pmatrix}
\tilde{\phi}^{m-r}(\sigma)\tilde{\phi}^{r}(\iota\sigma\iota)& 0\\
0&\tilde{\phi}^r(\sigma)\tilde{\phi}^{m-r}(\iota\sigma\iota)
\end{pmatrix}
=\tilde{\phi}^{r}(\sigma\iota\sigma\iota)\begin{pmatrix}
\tilde{\phi}^{m-2r}(\sigma)& 0\\
0&\tilde{\phi}^{m-2r}(\iota\sigma\iota)
\end{pmatrix},
$$
whereas that of $\sigma=\iota\sigma'$ with $\sigma'\in G_K$ is given by
$$
\begin{pmatrix}
0&\tilde{\phi}^r(\sigma')\tilde{\phi}^{m-r}(\iota\sigma'\iota)\\
\tilde{\phi}^{m-r}(\sigma')\tilde{\phi}^{r}(\iota\sigma'\iota)&0
\end{pmatrix}
=\tilde{\phi}^{r}(\sigma'\iota\sigma'\iota)\begin{pmatrix}
 0&\tilde{\phi}^{m-2r}(\iota\sigma'\iota)\\
\tilde{\phi}^{m-2r}(\sigma')&0
\end{pmatrix}.
$$
Therefore, we see that $\rho_r\cong\IndQK(V(\phi^{m-2r}))\cdot(\varepsilon_K\det\rho_f)^{r}$ by Lemma~\ref{det}.

Finally, for $r=m/2$, we have
$$
\sigma(x_{m/2})=
\begin{cases}
\tilde{\phi}^{m/2}(\sigma\iota\sigma\iota)x_{m/2}&\text{if $\sigma\in G_K$}\\
\tilde{\phi}^{m/2}(\sigma'\iota\sigma'\iota)x_{m/2}&\text{if $\sigma=\iota\sigma'$ where $\sigma'\in G_K$}.
\end{cases}
$$
Hence, $V_{m/2}=(\varepsilon_K\det\rho_f)^{m/2}$ again by Lemma~\ref{det}. This finishes the proof.
\end{proof}

\begin{corollary}
The complex $L$-function admits a factorisation
$$
L(\Sym^mf,s)=\begin{cases}
\left(\prod_{i=0}^{m/2-1} L\left(\phi^{m-2i},(\varepsilon_K\epsilon)^i,s-i(k-1)\right)\right)L\left((\varepsilon_K\epsilon)^{m/2},s-m/2(k-1)\right) &\text{if $m$ is even,}\\
\prod_{i=0}^{(m-1)/2} L\left(\phi^{m-2i},(\varepsilon_K\epsilon)^i,s-i(k-1)\right) &\text{otherwise.}
\end{cases}
$$
\end{corollary}
\begin{proof}
This can be proved in the same way as Corollary~\ref{Lfact}.
\end{proof}

\begin{remark}
For $0\le i \le \lfloor (m-1)/2\rfloor$, we may obtain a $p$-adic $L$-function that interpolates the $L$-values of $\phi^{m-2i}$ at $(m-2i)(k-1)$ using Proposition~\ref{interpo}. However, when $m>2$, their product does not interpolate the $L$-values of $\Sym^mf$. We would need $p$-adic $L$-functions that interpolate the $L$-values of $\phi^{m-2i}$ at $(m-i)(k-1)$ instead.
\end{remark}

\section*{Acknowledgement}
The author is extremely grateful for very helpful discussions with Daniel Delbourgo in the duration of the writing of this paper. He is indebted to Alex Bartel and Robert Harron for their useful comments on earlier versions of this paper. He is also indebted to an anonymous referee for many suggestions which helped to improve the paper.

\bibliographystyle{amsplain}
\bibliography{references}
\end{document}